\documentclass[11pt]{amsart}
\usepackage{amsmath,amssymb,amsthm,amscd}
\numberwithin{equation}{section}

\def\R{{\mathfrak R}}

\newtheorem{prop}{Proposition}[section]
\newtheorem{theo}{Theorem}
\newtheorem{lemm}[prop]{Lemma}
\newtheorem{coro}[prop]{Corollary}
\newtheorem{rema}[prop]{Remark}
\newtheorem{exam}[prop]{Example}
\newtheorem{defi}[prop]{Definition}

\def\begeq{\begin{equation}}
\def\endeq{\end{equation}}
\def\p{\partial}

\def\lf{\left}
\def\ri{\right}
\def\e{\epsilon}

\def\R{\Bbb R}

\def\Si{\Sigma}

\def\diam{{\rm diam}}
\def\Area{{\rm Area}}
\def\div{{\rm div}}

\begin{document}
\title{ On the behavior of quasi-local mass at the infinity along nearly round surfaces }
\author{Yuguang Shi$^1$}
\address{Key Laboratory of Pure and Applied mathematics, School of
 Mathematics Science, Peking University,
Beijing, 100871, P. R. China.}\email{ygshi@math.pku.edu.cn}
\author{Guofang Wang}

\address{Faculty of Mathematics, University Magdeburg, D-39016, Magdebrug, Germany}
\email{gwang@math.uni-magdeburg.de}
\thanks{$^1$Research partially supported by Grant of NSFC (10725101) and by   973 Program (2006CB805905)
 of China }

\author{Jie Wu$^1$}
\address{Key Laboratory of Pure and Applied mathematics, School of
 Mathematics Science, Peking University,
Beijing, 100871, P. R. China.} \email{wujie@math.pku.edu.cn}
\renewcommand{\subjclassname}{%
  \textup{2000} Mathematics Subject Classification}
\date{May 2008}
\begin{abstract}
In this paper,  we  study the limiting behavior  of the Brown-York
mass  and Hawking mass along nearly round surfaces at infinity of an
asymptotically flat manifold.  Nearly round surfaces can be defined
in an intrinsic way. Our results show that the ADM mass of an
asymptotically flat $3$-manifold can be approximated by some
geometric invariants of a family of nearly round surfaces which
approach to infinity of the manifold.
\end{abstract}

\maketitle\markboth{ Yuguang Shi, Guofang Wang and Jie Wu}{Infinity
limits of Quasilocal mass}

\section{Introduction}
The ADM mass  of an asymptotically flat (AF) manifold is a basic
conserved quantity in General relativity. To state its explicit
definition, we need the following:

\begin{defi}\label{defaf}
A complete three manifold $(M,g)$ is said to be {\rm asymptotically flat}
(AF) of order $\tau$ (with one end) if there is a compact subset $K$
such that $M\setminus K$ is diffeomorphic to $\R^3\setminus B_R(0)$
for some $R>0$ and in the standard coordinates in $\R^3$, the metric
$g$ satisfies:
\begin{equation} \label{daf1}
g_{ij}=\delta_{ij}+\sigma_{ij}
\end{equation}
with
\begin{equation} \label{daf2}
|\sigma_{ij}|+r|\p \sigma_{ij}|+r^2|\p\p\sigma_{ij}|=O(r^{-\tau}),
\end{equation}
for some constant $1\ge\tau>\frac{1}{2}$, where $r$ and $\p$ denote
the Euclidean distance and standard derivative operator on $\R^3$
respectively.
\end{defi}

A coordinate system of $M$ near infinity so that the metric tensor
in these coordinates satisfies the decay conditions in the
definition is said to be {\it admissible}, in such a coordinate
system, we have

\begin{defi}
The Arnowitt-Deser-Misner (ADM) mass (see \cite{ADM}) of an
asymptotically flat manifold $M$ is defined as:
\begin{equation} \label{defadm1}
m_{ADM}(M)=\lim_{r\to\infty}\frac{1}{16\pi}\int_{S_r}
\lf(g_{ij,i}-g_{ii,j}\ri)\nu^jd\sigma^0,
\end{equation}
where $S_r$ is the Euclidean sphere, $d\sigma^0$ is the volume
element induced by the Euclidean metric, $\nu$ is the outward unit
normal of $S_r$ in $\R^3$ and the derivative is the ordinary
partial derivative.
\end{defi}
We  always assume that the scalar curvature is in $L^1(M)$ so that
the limit exists in the definition. Under the decay conditions in
the definition of AF manifold, the definition of the ADM mass  is
independent of the choice of admissible coordinates  by the result
of Bartnik \cite{BTK86}. Indeed $S_r$ in the above definition does
not need to be the Euclidean sphere in some admissible coordinates,
it could be   a connected boundary of an exhausting domain with its
area growth like $r^2$. Here $r= \min_{x\in \Si} r(x)$, $r(x)$ is
the  distance function to some fixed point(see Proposition 4.1,
\cite{BTK86}). Hence the ADM mass of an AF manifold is actually a
geometric quantity. With these facts in mind and in the view point
of geometry, one may intend to ask: {\it Whether or not one can
define certain geometric invariants on  a family of surfaces defined
in an intrinsic way, i.e. are independent of the choice of
admissible coordinates,  that tends to the ADM mass as the surfaces
approach to the infinity of the manifold ?} In this paper, we will
investigate this problem and give an affirmative answer to it.

Intuitively, the ADM mass is a kind of total mass of $(M, g)$. In many
cases, we want to measure how much mass is contained in a bounded
domain. For this purpose, the notion of quasi-local energy (mass) is
needed. The Brwon-York mass and the Hawking mass are two of them which are
used frequently in literature and both of them are geometric
invariants of the surfaces (see the definitions below). Physically,
one natural property of quasi-local mass need to have is : {\it the
limit of quasi-local mass of the boundary of exhausting domains of
an AF manifold should approach the ADM mass }(see \cite{CY}). Many
people have studied this problem, they verified that for boundary of
certain exhausting domains this property is true for the Brown-York mass
and the Hawking mass, see \cite{BLP, BBWYY,BY2, HKGH1, Ma}, see also
\cite{FST, ST02}. However, the definitions of these boundaries
considered in above mentioned  papers depend on some special coordinates.
So, these surfaces are not intrinsic.

In this paper, we will discuss the problem  mentioned above in the
case that surfaces are {\it nearly round } at infinity of  an AF
manifold $(M, g)$. Let us begin with the following definition

\begin{defi}
Let $\{\Si_r\}$ be a family of surfaces which are topological
sphere in $(M, g)$, $r= \min_{x\in \Si} r(x)$,  then we call
$\Si_r$ as nearly round  when $r$ tends to infinity if it
satisfies:
\begin{enumerate}
    \item $|\overset\circ A| +r |\nabla \overset\circ A| \leq C r^{-1-\tau}$,
    \item $\max_{x\in \Si_r} r(x) \leq C\min_{x\in \Si_r} r(x)+C$,
    \item $\diam (\Si_r) \leq C r$,
    \item $\Area (\Si_r)\leq C r^2$.
\end{enumerate}
Here  $C$  is a constant independent of $r$. $\diam (\Si_r)$,
$\nabla$, and $|\cdot|$, denote  diameter of the surface, covariant
derivatives and the norm with respect to the induced metric of $g$
respectively, $r(x)$ is the distance of $x$ to some fixed point in
$(M,g)$, $A$ is the second fundamental forms of $\Si_r$ in $(M,g)$ and
$\overset\circ A$ is the trace free part of $A$.
\end{defi}

\begin{rema}
\begin{enumerate}
\item It is easy to see that the above definition of nearly round
surface is intrinsic, i.e. it does not depend on any coordinates.
  \item We suspect that the third and the fourth assumptions are
superfluous, since both of them can be derived from the first and
second assumptions in the Euclidean space case.
  \item It is not difficult to see that the third assumption
implies the second one.
\end{enumerate}
\end{rema}

One of very important and also quite natural surfaces in AF
manifolds are those with constant mean curvature and approach to
infinity of the manifolds, the existence of these surfaces was
proved by \cite{Ye}, and \cite{HY} many years ago, and later the
uniqueness was obtained by \cite{QT}(see also \cite{HY}). It is not
so difficult to see that these constant mean curvature surfaces are
nearly round (see the discussion at the beginning of Section 2).
Besides this, all the surfaces considered in \cite{BLP, BBWYY,BY2,
HKGH1, Ma,FST, ST02} are nearly round.

Now, let us move to the definition of the Brown-York mass and the Hawking
mass.

Let $\lf(\Omega, g\ri)$ be a compact three manifold with smooth
boundary $\p \Omega$. Suppose the Gauss curvature of $\p\Omega$ is
positive, then the Brown-York quasi local mass of $\p\Omega$ is
defined as  (see \cite{BY1,BY2}):
\begin{defi}
\begin{equation} \label{defbym1}
m_{BY}\lf(\p \Omega\ri)=\frac{1}{8\pi}\int_{\p \Omega}(H_0-H)d\sigma
\end{equation}
 where  $H$ is the mean curvature of $\p\Omega$ with respect to the
 outward unit normal and
  the metric $g$, $d\sigma$ is the volume element induced on $\p\Omega$ by
 $g$ and $H_0$ is the mean curvature  of $\p \Omega$ when
  embedded in  $\R^3$.
\end{defi}
The Brown-York mass is well-defined because by the result of
Nirenberg \cite{Nr}, $\p\Omega$ can be isometrically embedded in
$\R^3$ and the embedding is unique by
\cite{Herglotz43,Sacksteder62,PAV}. In particular, $H_0$ is
completely determined by the metric on $\p\Omega$. However, this is
a global property.  In contrast, the norm of the mean curvature
vector of an embedding of $\p\Omega$ into the light cone in the
Minkowski space can be expressed explicitly in terms of the Gauss
curvature, see \cite{BLY99}. Hence in the the study of Brown-York
mass,   one of the difficulties  is to estimate
$\int_{\p\Omega}H_0d\sigma$. We will use the Minkowski formulae
\cite{KLBG} and the estimates of Nirenberg \cite{Nr} to deal with
this problem.

The Hawking quasi local mass is  defined as (see \cite{HKG}):
\begin{defi}
\begin{equation} \label{defhkgm1}
m_H(\p \Omega)=\frac{|\p
\Omega|^{1/2}}{(16\pi)^{3/2}}\lf(16\pi-\int_{\p
\Omega}H^2d\sigma\ri)
\end{equation}
 where $d\Sigma$ is the volume element induced on $\p\Omega$ by
 $g$
 and $|\p \Omega|$ is the area of $\p \Omega$.
\end{defi}

Our main results in this paper are:

\begin{theo}\label{BYmass}
Let $(M,g)$ be an AF manifold, $\{\Si_r\}$ be  nearly  round
surfaces of $(M, g)$ when $r$ tends to infinity, then
$$
\lim_{r\rightarrow \infty}m_{BY}(\Si_r)=m_{ADM}(M).
$$
\end{theo}

\begin{theo}\label{Hawkingmass}
Let $(M,g)$ be an AF manifold, $\{\Si_r\}$ be  nearly  round
surfaces of $(M, g)$ when $r$ tends to infinity then
$$
\lim_{r\rightarrow \infty}m_{H}(\Si_r)=m_{ADM}(M).
$$
\end{theo}

The remaining of this paper is organized  in the following way. In  Section 2 we
will discuss the  geometry of nearly round surfaces,
show that many interesting surfaces are nearly round and  present
some useful formulae. In Section 3 we will show some estimates of
isometric embedding and  in Section 4 we will prove the main theorems.

\section{Geometry of nearly round  surfaces }

In this section, we want to give some examples of nearly round
surfaces and to investigate their geometric properties,  and  also
we will derive some basic formulae which will be used later. Let us
begin with some interesting examples.

\begin{exam}
Constant mean curvature (CMC) surfaces constructed  in \cite{HY}
are nearly round at the infinity of the manifold .
\end{exam}

Note that the CMC surfaces constructed in \cite{HY} are convex at
infinity, and then by Propositions 3.5, 3.9 and 3.12 in \cite{HY} we
see that they are nearly round at infinity of the manifolds.

Also, by a direct computations, it is not difficult to see that

\begin{exam}
Let $(M,g, x^i)$ , $1\leq i\leq 3$, be an AF manifold with
admissible coordinates $x^i$, then the coordinate sphere
$S_r=\{(x^1, x^2, x^3)|$\quad$ (x^1)^2+(x^2)^2+(x^3)^2=r^2\}$ is
nearly round when $r$ tends to infinity.
\end{exam}

Our next example relates to the surfaces in  Kerr solution to vacuum
Einstein equations.

\begin{exam}
The well-known  Kerr
metric is given by
\begin{equation}
\begin{split}
ds^2 =&-(1-\frac{2mr}{r^2+ a^2 \cos^2 \theta})dt^2 -\frac{2ma r \sin
^2
\theta}{r^2+ a^2 \cos^2 \theta}(dt d\phi +d\phi dt)\\
&+\frac{r^2+ a^2 \cos^2 \theta}{r^2 -2m r +a^2}dr^2 + (r^2+ a^2
\cos^2)d\theta^2 \\
&+\frac{\sin^2 \theta}{(r^2+ a^2 \cos^2 \theta)^2}[(r^2 +a^2)^2 -
a^2 (r^2-2mr +a^2) \sin^2 \theta]d\phi^2.
\end{split}
\end{equation}
For instance, see page 261 in \cite{Ca}.
Let $(M, g)$ be the slice  with $t= const.$, then it can be shown
directly that $(M, g)$ is an AF manifold. Let $\Si_\tau$ be the
surface in $(M,g)$ with $r=\tau$, then one can verify that $\Si_\tau$
is nearly round in $(M, g)$  when $\tau$ goes to infinity. Indeed,
if we let $e_1 = (\tau^2+ a^2 \cos^2)^{-\frac12} \frac{\p}{\p
\theta}$, $e_2= \frac{\tau^2+ a^2 \cos^2 \theta}{\sin
\theta}[(\tau^2 +a^2)^2 - a^2 (\tau^2-2m\tau +a^2) \sin^2
\theta]^{-\frac12}\frac{\p}{\p \phi}$, then it is easy to see that
$e_1$, $e_2$ is an  orthonormall frame of $\Si_\tau$, and in  this
frame the second fundamental form of $\Si_\tau$
with respect to outward unit normal vector in $(M,g)$ is
$$
A_{11}=\frac{2}{\tau}
\lf(\frac{1-\frac{2m}{\tau}+\frac{a^2}{\tau^2}}{1+\frac{a^2}{\tau^2}\cos^2
\theta}\ri)^\frac12 (1+\frac{a^2}{\tau^2}\cos^2 \theta)^{-1},
\quad
A_{12}=0,
$$
\begin{equation}
\begin{split}
A_{22}=& \frac{2}{\tau^2}
 \lf(\frac{1-\frac{2m}{\tau}+\frac{a^2}{\tau^2}}{1+\frac{a^2}{\tau^2}\cos^2
\theta}\ri)^\frac12 \lf(\frac{ 1+\frac{a^2}{\tau^2}\cos^2 \theta
}{(1+\frac{a^2}{\tau^2})^2 -\tau^{-2}
(1-\frac{2m}{\tau}+\frac{a^2}{\tau^2})a^2\sin^2 \theta}\ri)\\
&\cdot \frac{1}{(\tau^2 +a^2 \cos^2 \theta)^2} \left\{\tau^5 +2a^2
\tau^3 \cos^2 \theta -a^2 \tau^2 m \sin^2 \theta\right.
\\ & \left. + a^4 \tau (\cos
2\theta +\sin^4
\theta)
+m a^4 \sin^2 \theta \cos^2 \theta \right\}.
\end{split}\nonumber
\end{equation}
When $\tau$ tends to infinity, we have
$$
A_{11}=\frac{2}{\tau}-\frac{2m}{\tau^2} +O(\tau^{-3}),
\quad
A_{12}=0,
$$
$$
A_{22}=\frac{2}{\tau}-\frac{2m}{\tau^2} +O(\tau^{-3}).
$$
Hence
$$
|\overset\circ A|\leq C \tau^{-3}.
$$
Similarly  we  have
$$
|\nabla \overset \circ A|\leq C \tau^{-4}.
$$
Here $C$ is a constant  independent of $\tau$. Hence, $\Si_\tau$
is a nearly round surface when $\tau$  tends to infinity.

\end{exam}

\begin{lemm}\label{estimate of area}
Let $\Si_r$ be  nearly round  surfaces in $(M, g)$  when $r$ goes to
infinity, then there is a positive constant $\Lambda$ which is
independent of $r$, and with $\Lambda^{-1} r^2 \leq \Area
(\Si_r)\leq \Lambda r^2$.
\end{lemm}
\begin{proof}

It suffices to show the lower bound of the area. Since $g$ is an AF
metric, without loss of generality we may assume $g$ and $\hat g$
are equivalent on $M\setminus K$. Here and in the sequel, $\hat g$ is
background Euclidean metric on $M\setminus K$. Thus, we only need to
show $\Area (\Si_r, \hat g)\geq C r^2$, where $\Area (\Si_r, \hat g)$
is the area of $\Si_r$ with respect to metric $\hat g$. From the
second assumption of nearly round sphere surfaces, we know that the
standard sphere with radius $\frac{r}{2}$, denoted by
$\mathbb{S}^2 _{\frac{r}{2}}$, is in the domain enclosed by
$\Si_r$. Let $\Omega$ be the domain enclosed by $\Si_r$ and
$\mathbb{S}^2 _{\frac{r}{2}}$, $X=(X^1, X^2, X^3)$ be the Euclidean
coordinates. In view of
$$
\div_{\hat g} (\frac{X}{|X|})=\frac{2}{|X|}>0,
$$
we have
\begin{equation}
\begin{split}
\int_{\Si_r}\frac{X}{|X|}\cdot n -\int_{\mathbb{S}^2
_{\frac{r}{2}}}\frac{X}{|X|}\cdot n &= \int_\Omega \div_{\hat g}
(\frac{X}{|X|})\\
&\geq 0,
\end{split}\nonumber
\end{equation}
where $n$ is the outward unit normal vector of on one part  of the boundary of $\Omega$, $\Si_r$
and $-n$ is the  outward unit normal vector of on another part of the boundary of $\Omega$, $\mathbb{S}^2 _{\frac{r}{2}}$.
Therefore, we see that
$$
\Area (\Si_r, \hat g)\geq \pi r^2.
$$
This finishes the proof of the Lemma.
\end{proof}

Our next lemma is on the estimation of the decay of the second
fundamental forms of nearly round sphere surfaces.

\begin{lemm}\label{decay2ndforms}
Let $\Si_r$ be  nearly round  surfaces in $(M, g)$ when $r$ goes to
infinity, then we have
$$
|A|\leq C r^{-1},
$$
where $C$ is  a positive  constant independent of $r$.
\end{lemm}

\begin{proof}

Let $e_0$, $e_1$ and $e_2$ be the orthonormall frame of $(M,g)$ at
any fixed point of $\Si_r$, $e_1$ and $e_2$ be the tangential
vectors of $\Si_r$. Let $\nabla_k A_{ij}$ be the components of
$\nabla A$, then by a direct computation for $1\leq$ $i$, $j$, $k$
$\leq 2$,  we have
$$
\nabla_i A_{jk} =E_{ijk}+F_{ijk},
$$
where
$$
E_{ijk}=\frac14(\nabla_i H \bar g_{jk}+\nabla_j H \bar
g_{ik}+\nabla_k H \bar g_{ij})-\frac12 w_i \bar g_{jk}+\frac12 (w_j
\bar g_{ik}+w_k \bar g_{ij}).
$$
Here $w_i = R_{0kil}\bar g^{kl}$ and $\bar g_{ij}$ is the induced
metric on $\Si_r$. By the Codazzi equations, we have

$$
\langle E_{ijk}, F_{ijk}    \rangle=0,
$$
and
$$
|E_{ijk}|^2 =\frac35 |\nabla H|^2 +|w|^2 -\langle w_i, \nabla_i H
\rangle.
$$
thus combining these equalities  we have
\begin{equation}
\begin{split}
|\nabla \overset {\circ}A|^2 &=|\nabla A|^2 -\frac12 |\nabla H|^2
 \geq |E|^2-\frac12 |\nabla H|^2 \\
&\geq \frac{1}{20}|\nabla H|^2 -C |w|^2
\geq \frac{1}{200}(|\nabla A|^2- |\nabla \overset {\circ}A|^2)
-C|w|^2,
\end{split}
\end{equation}
where $C$ is a  constant. Thus, we see that
$$
|\nabla A| \leq C (|\nabla \overset\circ A|+  r^{-2-\tau}).
$$
Combining this  with the first assumption of nearly round sphere
surfaces we have
$$
|\nabla H |\leq |\nabla A| \leq  C r^{-2-\tau}.
$$
Due to the assumption $\diam (\Si_r) \leq C r$, we see that for any
$x\in \Si_r$ we have

\begin{equation}
\begin{split}
|H(x)-H(x_0)|&\leq |\nabla H| \cdot \diam (\Si_r)\\
&\leq C r^{-1-\tau}.\nonumber
\end{split}
\end{equation}
Here $x_0$ is a fixed point on $\Si_r$.
Setting
$$
r_1 =\frac{2}{H(x_0)},
$$
we have $$
H=\frac{2}{r_1}+O(r^{-1-\tau}).
$$
We now claim that there is a    constant $C> 1$ independent of $r$ with
$$
C^{-1} r \leq r_1 \leq C r.
$$
Indeed, due to the assumption $|\overset\circ A| \leq C
r^{-1-\tau}$ and the Gauss equations, we have
$$
K=\frac{1}{r^2 _1}+r^{-1} _{1} O(r^{-1-\tau}) +O(r^{-2-\tau}) ,
$$
where $K$ is the Gauss curvature of $\Si_r$ with  respect to the metric induced
from $g$. Then by the Gauss-Bonnet formula we get

$$
\frac{\Area (\Si_r)}{r^{2} _{1}}+r^{-1} _{1} O(r^{1-\tau})
+O(r^{-\tau})=4\pi.
$$
Together with Lemma \ref{estimate of area},  it follows that the claim
is true. Note that $A_{ij}=\overset\circ A_{ij}
+\frac{H}{2}\bar{g}_{ij}$, where $\bar{g}_{ij}$ is the  metric on
$\Si_r$ induced from $g$. The first assumption and the claim imply
$$
|A|\leq C r^{-1},
$$
which finishes the proof of  the lemma.
\end{proof}

As mentioned before, we may regard $\Si_r$ as a surface in
$M\setminus K$ with the Euclidean  metric $\hat g$. Our next lemma
is about the relationship between $A$ and $\hat A$, here and in the
sequel, $\hat A$ is the second fundamental forms of $\Si_r$ with
respect to outward unit normal vector and metric $\hat g$.

\begin{lemm}\label{relation2ndform}
Let $\rho$ be the Euclidean distance function to $\Si_r$,
$\frac{\p}{\p x^i}$, $1\leq i \leq 3$, be the standard coordinate
frame in $\mathbb{R}^3$, $\Gamma^k _{ij}$ is the Christoffel
symbols of metric $g$ with respect to $\frac{\p}{\p x^i}$, then
$$
\hat A(X,Y)=|\nabla_g \rho|\cdot A(X,Y)+X^i Y^j\Gamma^k
_{ij}\frac{\p \rho}{\p x^k},
$$
 where $X=X^i\cdot \frac{\p}{\p x^i}$ and $Y=Y^i\cdot \frac{\p}{\p x^i}$ are tangential vectors of
 $\Si_r$ and $\nabla_g$ is the Livi-Civita connection with respect to $g$.
\end{lemm}

\begin{proof}
By the definition of $\rho$,  on $\Si_r$ we have
\begin{equation}
\begin{array}{rcl}
\nabla^2 \rho (X, Y)&=&XY(\rho)-\nabla_X Y (\rho)\\
&=&-\nabla_X Y (\rho)\\
&=& -\langle\nabla_X Y , v  \rangle v(\rho)\\
&=& A(X,Y)v(\rho),
\end{array}\nonumber
\end{equation}
where $v$ is the outward unit normal vector of $\Si_r$. Again,  a
direct computations gives
$$
v=|\nabla_g \rho|^{-1} g^{ij}\frac{\p \rho}{\p x^i}\frac{\p}{\p
x^j},
$$
Hence, we have
$$
v(\rho)=|\nabla_g \rho|,
$$
$$
A(X,Y)=\frac{1}{|\nabla_g \rho|}\nabla^2 \rho (X, Y)
$$
and
$$
\nabla^2 \rho (X, Y)=X^i Y^j \frac{\p^2 \rho}{\p x^i \p x^j}-X^i
Y^j\Gamma^k _{ij}\frac{\p \rho}{\p x^k}.
$$
Similarly  we have

$$
\hat A(X,Y)=X^i Y^j \frac{\p^2 \rho}{\p x^i \p x^j },
$$
Combining these formulas together, we have
$$
\hat A(X,Y)=|\nabla_g \rho|\cdot A(X,Y)+X^i Y^j\Gamma^k
_{ij}\frac{\p \rho}{\p x^k}.
$$
\end{proof}

Let $v=\sum_{i=1}^3 v^i \frac{\p}{\p x^i} $, then $h_{ij}=g_{ij}-v_i
v_j$, $1\leq i$, $j\leq 3$. Then $h_{ij}dx_idx_j$ is the induced
metric on $\Si_r$ in $(M, g)$. Define $h^{ij}:=g^{is}g^{jt}h_{st}$.
We have
\begin{lemm}\label{2ndfundamentalforms}
Let $\rho$ be the  Euclidean distance function to $\Si_r$, then on
$\Si_r$, we have
$$
\frac{\p^2 \rho}{\p x^i \p x^j}= B_{ij} +\frac{\hat
H}{2}h_{ij},
$$
where $B_{ij}=\overset{\circ}{\hat A}((\frac{\p}{\p x^i})^T,
(\frac{\p}{\p x^j})^T )$, and $(\frac{\p}{\p x^i})^T$ is the
tangential part of $\frac{\p}{\p x^i}$.
\end{lemm}

\begin{proof} Let $\hat \nabla$ be the covariant derivatives with
respect to $\hat g$, then by a direct computation gives
\begin{equation}
\begin{split}
\frac{\p^2 \rho}{\p x^i \p x^j}& =\hat \nabla^2 \rho
((\frac{\p}{\p
x^i})^T, (\frac{\p}{\p x^j})^T )\\
&=\hat A ((\frac{\p}{\p x^i})^T, (\frac{\p}{\p x^j})^T )\\
&=\overset{\circ}{\hat A}((\frac{\p}{\p x^i})^T, (\frac{\p}{\p
x^j})^T ) +\frac{\hat H}{2}h_{ij}.
\end{split}
\end{equation}

\end{proof}

Combining Lemma (\ref{decay2ndforms}), Lemma (\ref{relation2ndform})
and Lemma (\ref{2ndfundamentalforms}), we obtain

\begin{coro}\label{estimate of 2ndderivative}

Let $\Si_r$ be   nearly round  surfaces in $(M, g)$ as $r$ goes to
infinity, then on $\Si_r$, we have:

$$
|\frac{\p^2 \rho}{\p x^i \p x^j}|\leq C r^{-1},
$$
where $C$ is a  constant independent of $r$.
\end{coro}

As a corollary, we have

\begin{coro}\label{decay2ndformsinR3}
Let $\overset{\circ}{\hat A}$ be the trace free part of $\hat A$,
$\bar D$ be the covariant derivatives of $\Si_r$ with respect to induced metric
from $(M\setminus K, \hat g)$, then we have
$$
|\overset{\circ}{\hat A}|+r |\bar D \overset{\circ}{\hat A}|\leq
Cr^{-1-\tau},
$$
where $C$ is a positive constant  independent of $r$.
\end{coro}
\begin{proof}
By Lemma \ref{relation2ndform}  and direct computations, we have
$$
|H-\hat H|\leq Cr^{-1-\tau},
$$
where $C$ is a positive constant  independent of $r$, and hence
$$
|\overset{\circ}{\hat A}|\leq Cr^{-1-\tau}.
$$
Hence it suffices to show the second part estimate of the corollary is
true. Let $p$  be any point of $\Si_r$, $e_1$ and $e_2$ be the
orthonormal frame at $p$   with $\bar D_{e_i} e_j =0$, $1\leq
i$, $j\leq 2$. Let $X_i$, $1\leq i\leq 3$, be one of $e_k$.  Then,
at $p$, we have
\begin{equation}
\begin{split}
(\bar D_{X_3} \hat A)(X_1, X_2)=&X_3 (\hat A(X_1, X_2))\\
=&X_3 (|\nabla_g \rho|)\cdot A(X_1,X_2)+|\nabla_g \rho|X_3
(A(X_1,X_2))\\
&+(X_3 (X^i _1)X^j _2 + X^i _1 X_3 (X^j _2))\Gamma^k _{ij}\frac{\p
\rho}{\p
x^k}+ X^i _1 X^j _2 X_3 (\Gamma^k _{ij})\frac{\p \rho}{\p x^k}\\
&+X^i _1 X^j _2 \Gamma^k _{ij}X_3 (\frac{\p \rho}{\p x^k}).
\end{split}
\end{equation}
Here we assume $X_i =X_i ^k \frac{\p}{\p x^k}$. Since
$\rho$ is Euclidean distance to $\Si_r$,  we have
$$
\sum_{i=1}^3 (\frac{\p \rho}{\p x^i})^2 =1.
$$
Thus,
\begin{equation}
\begin{split}
|\sum_{i=1} ^3 g^{ij} X_3 (\frac{\p \rho}{\p x^i})\frac{\p \rho}{\p
x^j}|&\leq |\sum_{i=1}^3 \sigma^{ij} X_3 (\frac{\p \rho}{\p
x^i})\frac{\p
\rho}{\p x^j}|+C r^{-1-\tau}\\
&\leq  Cr^{-1-\tau},
\end{split}\nonumber
\end{equation}
where  $C$ is a constant  independent of $r$ and $p$ and the
orthonormal frames that we choose. Note that   Corollary
\ref{estimate of 2ndderivative} was used  in the last equality.

Due to the asymptotically flatness of  manifold $(M,g)$, by Lemma (\ref{decay2ndforms}) and Corollary (\ref{estimate of
2ndderivative}) we have
$$
|X_3 (|\nabla_g \rho|)\cdot A(X_1,X_2)|+|X^i _1 X^j _2 X_3 (\Gamma^k
_{ij})\frac{\p \rho}{\p x^k}|+|X^i _1 X^j _2 \Gamma^k _{ij}X_3
(\frac{\p \rho}{\p x^k})|\leq C r^{-2-\tau},
$$
where   $C$ is a positive constant independent of $r$ and $p$ and
the orthonormal frames that we choose. In order to get the estimates
of the remaining part, we need to estimate covariant derivatives of
$X_i$ at $p$ first. Note that by Lemma \ref{relation2ndform}, we see
that there is a  constant $C$ which is independent of $r$  and with
$$
|A-\hat A|\leq C r^{-1-\tau}.
$$
On the other hand, by the fundamental equations of the surface, for
$i=1$, $2$,  we have
\begin{equation}
\begin{split}
 \bar{\nabla}_{X_3} X_i -\bar {D}_{X_3} X_i &=(\nabla_{X_3} X_i
-D_{X_3} X_i)+ (\hat A -A)(X_3, X_i)v +\hat A (X_3 , X_i)(\hat v -v)\\
&=X^i _3 X^k _i \Gamma^l _{ik} \frac{\p}{\p x^l}+(\hat A -A)(X_3,
X_i)v +\hat A(X_3, X_i)(\hat v-v),
\end{split}\nonumber
\end{equation}
where $\nabla$, $\bar \nabla$ is covariant derivatives with respect
to metric $g$ and its induced metric on $\Si_r$ respectively. By a
direct computations, it is not difficult to see that
$$
|\hat v- v|\leq C r^{-\tau},
$$
where $C$ is a positive universal constant independent of $r$.
Hence, by the choice of $X_i$ and decay of $|A-\hat A|$,  we
get
\begin{equation}\label{estimateof covariant derivative}
|\bar \nabla_{X_3} X_i|\leq C r^{-1-\tau},
\end{equation}
at $p$, where $C$ is a  constant independent of $r$  and $p$ and the
orthonormal frames that we choose. Together this with   decay of
$|\bar\nabla A|$ and the equality
$$
X_3 (A(X_1, X_2))=(\bar \nabla_{X_3} A)(X_i, X_j) + A(\bar
\nabla_{X_3} X_i, X_j)+A(\bar \nabla_{X_3} X_j, X_i),
$$
$i$, $j=$ $1$, $2$,  we get
$$
|X_3 (A(X_i, X_j))|\leq C r^{-2-\tau},
$$
where $C$ is a  constant independent of $r$  and $p$ and the
orthonormal frames that we choose. By the choice of $X_i$, we have
for $i=1$, $2$,
$$
\bar D_{X_3}X_i =D_{X_3}X_i -\hat A (X_3, X_i)\hat v=0,
$$
 at $p$,
 and hence
 $$
X_3 (X_i ^k)=\hat g (\hat A (X_3, X_i)\hat v, \frac{\p}{\p x^k})
 $$
at $p$
 which implies
 $$
|X_3 (X_i ^k)|\leq C r^{-1}
 $$
for a  constant $C$ independent of $r$  and $p$ and the orthonormal
frames that we choose. Combining the above estimates, we have
$$
|\bar D \hat A| \leq C r^{-2-\tau}
$$
at $p$ for some universal constant independent of $r$  and $p$ and
the orthonormal frames that we choose. Since $p$ is arbitrary, we
complete to prove the corollary.
\end{proof}

Let $H$ and $\hat H$ be the mean curvature of $\Si_r$ in $(\R^3
\setminus K, g)$ and  $(\R^3 \setminus K, \hat g)$ respectively.
Both of them are  with respect to outward unit normal vector.

\begin{lemm}\label{difference of mean curvature}
Let $\rho$ be the Euclidean distance to $\Si_r$ in $\R^3$, then
\begin{equation}\label{meancuravture}
\begin{split}
H=&\hat H +\frac{H}{2}\sigma_{ij}\frac{\p \rho}{\p x^i}\frac{\p
\rho}{\p x^j}+\frac12 \sigma_{st,i}\frac{\p \rho}{\p x^i}\frac{\p
\rho}{\p x^s}\frac{\p \rho}{\p x^t}\\
&-\sigma_{ij}\frac{\p^2 \rho}{\p x^i \p x^j}-g_{ij,i}\frac{\p
\rho}{\p x^j}
+\frac12 g_{jj,i}\frac{\p \rho}{\p x^i}+O(r^{-1-2\tau}).\\
\end{split}
\end{equation}
Here and in the sequel, $\sigma_{ij,k}=\frac{\p \sigma_{ij}}{\p
x^k}$ and $g_{ij,k}=\frac{\p g_{ij}}{\p x^k}$.

\end{lemm}

\begin{proof}
 We first note that on $\Si$ we have
$$
\Delta_0 \rho =\hat H,
$$
where $\Delta_0$ is Laplacian with respect to $\R^3$. The unit
normal vector of $\Si_r$ in $\R^3$ is $\nabla_0 \rho =\frac{\p
\rho}{\p x^i}\frac{\p}{\p x^i}$,  denoted by $\hat v$. Let $\{e_1,
e_2, v\}$ be the orthonormal frame in $(M\setminus K, g)$ and $\{e_1,
e_2\}$ the tangential vector of $\Si$. We have
$$
\Delta \rho = \nabla^2 \rho (e_1 , e_1) + \nabla^2 \rho (e_2 ,
e_2)+ \nabla^2 \rho (v , v)
$$
and
\begin{equation}
\begin{split}
 \nabla^2 \rho (e_1 , e_1) + \nabla^2 \rho (e_2 , e_2)&=e_1 e_1 \rho
 -\nabla_{e_1}e_1 \rho +e_2 e_2 \rho
 -\nabla_{e_2}e_2 \rho\\
 &=-(\nabla_{e_1}e_1 +\nabla_{e_2}e_2)\rho\\
 &=\langle\nabla_{e_1}e_1 +\nabla_{e_2}e_2, v\rangle v(\rho)\\
 &=H\cdot v(\rho).
\end{split}
\end{equation}
On the other hand, it is clear that
$$
\hat v = \langle \hat v, v\rangle_g v + T,
$$
where, $T$ is the tangential part of $\hat v$ on $\Si_r$ and
$\langle, \rangle_g$ is the inner product with respect to the metric
$g$. Thus we have
$$
\hat v(\rho)=\langle \hat v, v\rangle_g v(\rho).
$$
Since $\hat v (\rho)=1$ on $\Si_r$, we get
$$
v(\rho)=\langle \hat v, v\rangle_g ^{-1}.
$$
It is easy to see that
$$
v=\frac{\nabla \rho}{|\nabla \rho|}=|\nabla \rho|^{-1} g^{ij}
\frac{\p \rho}{\p x^i}\frac{\p}{\p x^j}
$$
and
\begin{equation}
\begin{split}
\langle \hat v, v\rangle_g&=|\nabla \rho|^{-1}g^{ij} \frac{\p
\rho}{\p x^i}\langle \frac{\p}{\p x^j},  \frac{\p}{\p
x^k}\rangle_g
\frac{\p \rho}{\p x^k}\\
&=|\nabla \rho|^{-1} g^{ij}\frac{\p \rho}{\p x^i}g_{jk}\frac{\p
\rho}{\p x^k}\\
&=|\nabla \rho|^{-1}.
\end{split}
\end{equation}
Hence, we have
$$
v(\rho)=|\nabla \rho|,\quad
\langle \hat v, v\rangle_g = |\nabla \rho|^{-1}.
$$
Since $v$ is the unit norm vector along $\Si_r$, it is clear that
$\langle \nabla_v v, v   \rangle_g =0$, which implies $\nabla_v v$
is a tangential vector on $\Si_r$. Therefore, we get
$$\nabla_v v (\rho)=0.$$
Combining these facts we obtain
$$
\nabla^2 \rho (v, v) = v(|\nabla \rho|)
$$
and
\begin{equation}\label{meancurvatureinafmfd}
\Delta \rho =H |\nabla \rho| + v(|\nabla \rho|).
\end{equation}

Now we compute  the second term in (\ref{meancurvatureinafmfd}).
\begin{equation}
\begin{split}
v(|\nabla \rho|^2)=&v(g^{st}\frac{\p \rho}{\p x^s}\frac{\p
\rho}{\p
x^t})\\
=&|\nabla \rho|^{-1}g^{ij}\frac{\p \rho}{\p x^i}\frac{\p}{\p
x^j}(g^{st}\frac{\p \rho}{\p x^s}\frac{\p \rho}{\p x^t})\\
=&|\nabla \rho|^{-1}g^{ij}\frac{\p \rho}{\p x^i}g^{st}_{,j}\frac{\p
\rho}{\p x^s}\frac{\p \rho}{\p x^t} +2|\nabla
\rho|^{-1}g^{ij}\frac{\p \rho}{\p x^i}g^{st}\frac{\p^2 \rho}{\p x^s
\p x^j}\frac{\p \rho}{\p x^t},
\end{split}
\end{equation}
where
$$
g^{st}_{, j}=\frac{\p g^{st}}{\p x^j}
$$
and
$$
g^{ij}=\delta^{ij}-\sigma_{ij}+O(r^{-2\tau}).
$$
Putting these things together, we get
\begin{equation}
\begin{split}
v(|\nabla \rho|^2)=&\frac{\p \rho}{\p x^i}g^{st}_{,i}\frac{\p
\rho}{\p x^s}\frac{\p \rho}{\p x^t}+2|\nabla \rho|^{-1}\frac{\p
\rho}{\p x^i}\frac{\p^2 \rho}{\p x^i \p x^j}\frac{\p \rho}{\p
x^j}\\
&-2|\nabla \rho|^{-1}\sigma_{ij}\frac{\p \rho}{\p x^i}\frac{\p^2
\rho}{\p x^s \p x^j}\frac{\p \rho}{\p x^s}-2|\nabla
\rho|^{-1}\frac{\p \rho}{\p x^i}\frac{\p^2 \rho}{\p x^i \p
x^s}\frac{\p \rho}{\p x^t}\sigma_{st}\\
&+ O(r^{-1-2\tau}).
\end{split}
\end{equation}
 From
$$
\sum_{i} \frac{\p \rho}{\p x^i}\frac{\p^2 \rho}{\p x^i \p
x^s}=\frac12 \sum_{i} \frac{\p}{\p x^s}((\frac{\p \rho}{\p
x^i})^2)=0
$$
and
$$
v(|\nabla \rho|^2)= 2 |\nabla \rho|v(|\nabla \rho|),
$$
 we get
 $$
v(|\nabla \rho|)=\frac12 g^{st}_{,i}\frac{\p \rho}{\p x^i}\frac{\p
\rho}{\p x^s}\frac{\p \rho}{\p x^t}+O(r^{-1-2\tau}).
 $$
Thus, we have
$$
\Delta \rho =H |\nabla \rho| +\frac12 g^{st}_{,i}\frac{\p \rho}{\p
x^i}\frac{\p \rho}{\p x^s}\frac{\p \rho}{\p x^t}+O(r^{-1-2\tau}).
$$
On the other hand, by the definition of Laplacian operator we
have
\begin{equation}
\begin{split}
\Delta \rho &=g^{ij} \frac{\p^2 \rho}{\p x^i \p
j}+\frac{1}{\sqrt
g}\frac{\p}{\p x^i}(\sqrt {g} g^{ij})\frac{\p \rho}{\p x^j}\\
&=\hat H -\sigma_{ij}\frac{\p^2 \rho}{\p x^i \p x^j}+
g^{ij}_{,i}\frac{\p \rho}{\p x^j}+g^{ij}\frac{1}{\sqrt
g}\frac{\p}{\p x^i}(\sqrt g)\frac{\p \rho}{\p x^j}.
\end{split}
\end{equation}
Noticing that
$$
\frac{\p}{\p x^i}(\sqrt g)= \frac12 g_{jj,i} + O(r^{-1-2\tau})
$$
and
$$
g^{ij}_{,i}=-g_{ij,i},
$$
we get
\begin{equation}\label{meancuravture}
\begin{split}
H=&\hat H +\frac{H}{2}\sigma_{ij}\frac{\p \rho}{\p x^i}\frac{\p
\rho}{\p x^j}+\frac12 \sigma_{st,i}\frac{\p \rho}{\p x^i}\frac{\p
\rho}{\p x^s}\frac{\p \rho}{\p x^t}\\
&-\sigma_{ij}\frac{\p^2 \rho}{\p x^i \p x^j}-g_{ij,i}\frac{\p
\rho}{\p x^j}
+\frac12 g_{jj,i}\frac{\p \rho}{\p x^i}+O(r^{-1-2\tau})\\
=&\hat H+O(r^{-1-\tau})+O(r^{-1-2\tau})
\end{split}
\end{equation}
\end{proof}

In the sequel, we want to calculate the integral of $ H-\hat H$ on
$\Si_r$. Let us begin with
\begin{equation}\label{3rdterminmeancurv}
\begin{split}
&\int_{\Si}\sigma_{st,i}\frac{\p \rho}{\p x^i}\frac{\p \rho}{\p
x^s}\frac{\p \rho}{\p x^t}d\sigma^0 =\int_{\Si}\frac{\p}{\p
x^i}(\sigma_{st}\frac{\p \rho}{\p x^s})\frac{\p \rho}{\p
x^i}\frac{\p \rho}{\p x^t}d\sigma^0
-\int_{\Sigma}\sigma_{st}\frac{\p^2 \rho}{\p x^i \p x^s}\frac{\p
\rho}{\p x^i}\frac{\p \rho}{\p x^t}d\sigma^0\\
&=-\int_{\Si}(\delta_{it}-\frac{\p \rho}{\p x^i}\frac{\p \rho}{\p
x^t})\frac{\p}{\p x^i}(\sigma_{st}\frac{\p \rho}{\p x^s})d\sigma^0
+\int_{\Si}\frac{\p}{\p x^t}(\sigma_{st}\frac{\p \rho}{\p
x^s})d\sigma^0\\
&=-\int_{\Si}\hat H \sigma_{st}\frac{\p \rho}{\p x^s}\frac{\p
\rho}{\p x^t}d\sigma^0 +\int_{\Si}\sigma_{st,t}\frac{\p \rho}{\p
x^s}d\sigma ^0 +\int_{\Si}\sigma_{st}\frac{\p^2 \rho}{\p x^s \p
x^t}d\sigma^0,
\end{split}
\end{equation}
where $d\sigma^0$ is the area element with respect to the Euclidean
induced metric, and we have used the divergence theorem in the
last equality.  By Lemma (\ref{estimate of area}), we have
\begin{equation}
\begin{split}
\int_{\Si}(H-\hat H)d \sigma =&\int_{\Si}(H-\hat H)d\sigma^0
+O(r^{1-2\tau})\\
=&\frac12 \int_{\Si}(H-\hat H)\sigma_{st}\frac{\p \rho}{\p
x^s}\frac{\p \rho} {\p x^t}d\sigma^0 +\frac12
\int_{\Si}(g_{ii,j}-g_{ij,i})\frac{\p
\rho}{\p x^j}d\sigma^0 \\
&-\frac12 \int_{\Si}\sigma_{st}\frac{\p^2 \rho}{\p x^s \p
x^t}d\sigma^0 +O(r^{1-2\tau}).
\end{split}
\end{equation}
Noticing that
$$H-\hat H = O(r^{-1-\tau})$$
we have
\begin{equation}\label{integral of diff meancurv}
\int_{\Si}(H-\hat H)d \sigma =\frac12
\int_{\Si}(g_{ii,j}-g_{ij,i})\frac{\p \rho}{\p x^j}d\sigma^0
-\frac12 \int_{\Si}\sigma_{st}\frac{\p^2 \rho}{\p x^s \p
x^t}d\sigma^0 +O(r^{1-2\tau}).
\end{equation}

\medskip

\section{Estimations of isometric embedding of nearly round   surfaces}

In this section, we study the isometric embedding of nearly round
surfaces, and the main purpose is to get the expansion of the mean
curvature nearly round   surfaces at the infinity of $(M,g)$.

 Let $(\mathbf{S}^2, g_0)$ be the
standard unit sphere and $i_0$ an isometric embedding of
$(\mathbf{S}^2, g_0)$ into $\mathbb{R}^3$. Let $K_g$ denote the Gauss curvature
of the metric $g$.
 We want to show the
following

\begin{theo}\label{estimateof isometric emb} There exists a positive constant
$\e_0>0$ such that for any metric $g$ on  $\mathbf{S}^2$ with
$$
\| K_g-1\|_{C^1} \leq \e_0,
$$
 there exists an isometric embedding
$$
i:(\mathbf{S}^2, g)\rightarrow \mathbb{R}^3
$$
and a conformal transformation $\Psi_1$ of $(\mathbf{S}^2, g_0)$
with
$$\|i\circ \Psi^{-1} _1 -i_0\|_{C^{2,\alpha}}\leq C_0 \| K-1\|_{C^\alpha}, $$
  for some $0<\alpha<1$. Here  $C_0$ is a positive  constant only depending on $\alpha$, and
$\|\cdot\|_{C^{2,\alpha}}$, $\|\cdot\|_{C^{\alpha}}$ are taken with respect to
$g_0$.
\end{theo}
\begin{proof}
The key point of the proof of above theorem is to show that there is a
conformal transformation $\Psi_1$ of the standard unit sphere with
\begin{equation}\label{perturbmetric}
\|\Psi_1 ^* (g)-g_0 \|_{C^{2,\alpha}} \leq C_1\| K-1\|_{C^\alpha},
\end{equation}
where $C_1$ is a  constant only depending on $\alpha$. Once
(\ref{perturbmetric}) is verified then by the arguments in P353 of
\cite{Nr}, we see that the conclusion of Theorem is true. Due to the
Uniformization Theorem, we see that there is conformal
differmorphism $\Phi$: $(\mathbf{S}^2, g_0)\mapsto (\mathbf{S}^2,
g)$ with $\Phi^* (g) = e^{2u} g_0$ and

\begin{equation}\label{equation1}
\Delta_{\mathbf{S}^2} u + K e^{2u}=1.
\end{equation}
Without loss of generality, we may assume for $1\leq i \leq 3$
\begin{equation}\label{masscenter}
\int_{\mathbf{S}^2}e^{2u} x_i =0,
\end{equation}
where $x_i$, $1\leq i \leq 3$, is the coordinate function of the
standard sphere in $\mathbb{R}^3$, integral is taken on the standard
sphere $(\mathbf{S}^2 , g_0)$. Otherwise, by Lemma 2, part 3 in
\cite{Ch} we can find a conformal transformation $\Phi_1$:
$(\mathbf{S}^2, g_0)\mapsto(\mathbf{S}^2, g_0)$ so that
$u\circ\Phi_1$ satisfying (\ref{equation1}), and (\ref{masscenter}).
Due to $a'$ in the proof of Theorem 1, Part 7 in \cite{Ch}, we know
that there a   constant $C$ only depending on $w$ and $W$ (here
$0<w\leq K \leq W$)  and with
\begin{equation}\label{compactness}
|u|_{C^0(\mathbf{S}^2)} \leq C.
\end{equation}
We now claim that for any $\eta>0$, there is $\delta=\delta(\eta)$
so that
\begin{equation}\label{small}
|u|_{C^0 (\mathbf{S}^2)}\leq \eta,
\end{equation}
provided $|K-1|_{C^0 (\mathbf{S}^2)}\leq \delta$.

Suppose the claim  fails, then we may find a constant $\eta_0
>0$, and a  sequence of $K_j$ and $u_j$ satisfying
(\ref{equation1}) and (\ref{masscenter}), and $|K_j -1|$ tends to
zero while $|u_j|_{C^0 (\mathbf{S}^2)}\geq \eta_0$. By
(\ref{compactness}) and the standard estimates in elliptic PDE,  we
may take a subsequence still denoted by $u_j$ which converges to
$u_\infty$ in the sense of $C^1(\mathbf{S}^2)$, and $u_\infty$
satisfying (\ref{masscenter}) and
$$
\Delta_{\mathbf{S}^2} u_\infty + e^{2u_\infty}=1.
$$
By this we see that $u_\infty =0$ which is contradiction to the
choice of $u_j$. Thus,  claim is true.

Let
$$u=u_0 + u_1 +u_2$$
with
$$
u_0=\frac{1}{4\pi}\int_{\mathbf{S}^2} u, \quad
u_1 =a_1 x_1 +a_2 x_2 +a_3 x_3, $$
where
$$
a_i =\frac{3}{4\pi} \int_{\mathbf{S}^2} u \cdot x_i
$$
for $1\leq 1\leq 3$.
Taking integral on (\ref{equation1}) on $(\mathbf{S}^2 , g_0)$, we
get
$$
\int_{\mathbf{S}^2}  K e^{2u} =4\pi.
$$
Together this with the assumption of $K$, we get
\begin{equation}\label{u0}
|u_0|\leq C (\|K-1\|_{C^0 (\mathbf{S}^2)}+\|u\|^2 _{L^2
(\mathbf{S}^2)}).
\end{equation}
Here  and in the sequel, $C$ is  always a  constant
independent of $u$.
Let
$$
L(u)
:=\Delta_{\mathbf{S}^2} u+ 2u,
$$
By equation (\ref{equation1}) and the definition of $u_2$,
we get
$$
L(u_2) = (1-K)e^{2u}+(1+2u -e^{2u})-2u_0,
$$
from which we have
\begin{equation}\label{estimate1}
\begin{split}
\int_{\mathbf{S}^2}( |\nabla u_2|^2 -2|u_2|^2 )\leq & C
(\|K-1\|_{C^0
(\mathbf{S}^2)}\|u_2\| _{L^2 (\mathbf{S}^2)}\\
& +\|1+2u -e^{2u}\|_{L^2 (\mathbf{S}^2)}\|u_2\| _{L^2
(\mathbf{S}^2)}\\
&+ |u_0|\|u_2\| _{L^2 (\mathbf{S}^2)}),
\end{split}
\end{equation}

From the
definition of $u_2$ we see that
\begin{equation}\label{estimate2}
\int_{\mathbf{S}^2}( |\nabla u_2|^2 -2|u_2|^2 )\geq C \|u_2 \|^2
_{L^2 (\mathbf{S}^2)}
\end{equation}
where $C$ is a positive universal constant  independent of $u$.
By (\ref{small}) we get
\begin{equation}\label{estimate3}
\|1+2u -e^{2u}\|_{L^2 (\mathbf{S}^2)} \leq C \|u\|^2 _{L^2
(\mathbf{S}^2)}
\end{equation}
Combining (\ref{estimate1}), (\ref{estimate2}), (\ref{estimate3})
and (\ref{u0}) we obtain
\begin{equation}\label{u2}
\|u_2 \| _{L^2 (\mathbf{S}^2)}\leq C (\|K-1\|_{C^0
(\mathbf{S}^2)}+\|u\|^2 _{L^2 (\mathbf{S}^2)})
\end{equation}
On the other hand, by the (\ref{masscenter}) and direct
computations we obtain, for each $i$,
$$
|a_i| \leq C \|u\|^2 _{L^2 (\mathbf{S}^2)},
$$
 which  implies
$$
\|u_1\| _{L^2 (\mathbf{S}^2)} \leq C  \|u\|^2 _{L^2
(\mathbf{S}^2)}.
$$

All the constants $C$ above are independent of $u$. Putting this
estimate with (\ref{u0}), (\ref{u2}), and (\ref{small}), we get
\begin{equation}\label{u}
\|u\| _{L^2 (\mathbf{S}^2)} \leq C \|K-1\|_{C^0 (\mathbf{S}^2)}.
\end{equation}
Together this with (\ref{equation1}), we get
$$
\|u \|_{W^{2,2}(\mathbf{S}^2)}\leq C \|K-1\|_{C^0 (\mathbf{S}^2)},
$$
which implies
$$
\|u \|_{C^\alpha(\mathbf{S}^2)}\leq C \|K-1\|_{C^0
(\mathbf{S}^2)},
$$
for some $\alpha>0$, here $C$ is a constant that only depends on
$\alpha$ and is  independent of $u$. Then by the Schauder theory in
partial differential equations, we get
$$
\|u \|_{C^{2,\alpha}(\mathbf{S}^2)}\leq C \|K-1\|_{C^0
(\mathbf{S}^2)},
$$
Thus, we see that
$$
\|\Psi^* _1 (g)-g_0 \|_{C^{2,\alpha}} \leq C_1\| K-1\|_{C^\alpha},
$$
where $C_1$ is a constant  depending only on $\alpha$. It
implies the conclusion of the Theorem is true. Thus, we finish to
prove the Theorem.
\end{proof}

Let $X$,  $n$ be the position vector and  the outward unit normal
vector of $i (\mathbf{S}^2)$ in $\mathbb{R}^3$ respectively,
$H_0$ be its mean curvature with respect to $n$, then  as a
corollary, we have
\begin{coro}\label{estimate2of isometric emb}
Let $(\mathbf{S}^2, g)$ be a two dimensional Riemannian manifold,
$K$ be its Gauss curvature. Then there exist a positive constant
$\e_0$ which is independent of $g$ such that if
$$
\| K-1\|_{C^1} \leq \e_0,
$$
then
$$|X\cdot n-1|\leq C \|K-1\|_{C^0
(\mathbf{S}^2)}, $$
and
$$
|H_0 -2|\leq C \|K-1\|_{C^0 (\mathbf{S}^2)},
$$
where $C$ is a constant that independent of $g$.
\end{coro}

\begin{proof}

By Theorem \ref{estimateof isometric emb}, we see that the statement
of the theorem is true for $i\circ \Psi^{-1}_1 (\mathbf{S}^2)$ in
$\mathbb{R}^3$. It is also true for the case of $i (\mathbf{S}^2)$,
for the support function $X\cdot n$ and the mean curvature $H_0$ are
independent of a parametrization of the domain manifold.
\end{proof}

On the other hand, by Corollary \ref{decay2ndformsinR3} and the
similar arguments as in Prop 2.1 in \cite{HY}, we may prove the
following

\begin{prop}\label{nearround}
Let $\hat \lambda_i$ be the principal curvature of $\Si$ in
$(M\setminus K, \hat g)$, If $|\AA|=O(r^{-1-\tau})$, $|\bar\nabla
\AA|=O(r^{-2-\tau})$, then there is a number $r_0 \in \R$ and a
vector $\overrightarrow{a} \in \R^3$ such that
$$
\hat \lambda_i -r^{-1}_0 =O(r^{-1-\tau}),
$$
$$
|(y-\overrightarrow{a})-r_0 \hat n|=O(r^{1-\tau}).
$$
Here $y$ is the position vector of $\Si$ in $(M\setminus K,\hat
g)$ (which is regarded as a subdomain  of $\R ^3$), $\hat n$ is
the outward unit normal vector of $\Si_r$.
\end{prop}

By the Gauss-Bonnet formula, Lemma \ref{estimate of area} and the
arguments we used before, we know that there is a  constant $C>0$
which is independent of $r$  and $r_0$ and  with $C^{-1} r \leq r_0
\leq C r$.

Now, we are in the position to  study the isometric embedding of
nearly round surfaces. Let $(M,g)$ be the AF manifolds and $\Si_r$
the nearly round surface in $(M,g)$ as $r$ goes to infinity. Then we
have

\begin{theo}\label{estimate H0}
Let $\Si_r\subset (M, g)$ be a nearly round  surfaces in $(M,g)$ as
$r$ goes to infinity,  and $r_0$  defined as in Proposition
\ref{nearround}. Then there is isometrically embedding $X_{r}$ of
$\Si_r$ into $\mathbb{R}^3$ such that
\begin{equation}\label{suppfun}
X_{r}\cdot n_0 =r_0 +O(r_0 ^{1-\tau})
\end{equation}
and
\begin{equation}\label{meancurvexp}
\|H_0  -\frac{2}{r_0}\|_{C^0}\leq C_3 r_0 ^{-1-\tau },
\end{equation}
provided that $r_0$ is  large enough. Here $n_0$ and $H_0 $ are the
unit outward normal vector and mean curvature of $X_r (\Si_r)$ in
$\mathbb{R}^3$ respectively. $C_3$ is a constant that is independent
of $r$.
\end{theo}
\begin{proof}
By the assumption on $A$ and $\bar\nabla A$, we see that
$$K=\frac{1}{r_0 ^2}+ O(r ^{-2-\tau})$$
and
$$
\|\bar\nabla K\|_{C^0}\leq C r ^{-3-\tau}.
$$
Then by a rescaling, we see that the resulting surface satisfying
the assumptions of Theorem \ref{estimateof isometric emb}. Then
combining Theorem \ref{estimateof isometric emb} with   direct
computations we get the conclusion.
\end{proof}

\section{Brown-York mass and Hawking  mass of nearly  round surfaces at infinity}

In this section, we prove our main results.
\begin{theo}\label{BYmass}
Let $(M,g)$ be an AF manifold, $\Si_r$ be a nearly  round surface of
$(M, g)$  as $r$ goes to infinity, then
$$
\lim_{r\rightarrow \infty}m_{BY}(\Si_r)=m_{ADM}(M).
$$
\end{theo}

\begin{theo}\label{Hawkingmass}
Let $(M,g)$ be an AF manifold, $\Si_r$ be a nearly  round surface of
$(M, g)$ as $r$ goes to infinity, then
$$
\lim_{r\rightarrow \infty}m_{H}(\Si_r)=m_{ADM}(M).
$$
\end{theo}

\begin{proof}[Proof of Theorem \ref{BYmass}]
Let $H_0$ be the mean curvature of the  isometric embedding image of
$(\Si_r, g)$ in $\R^3$. Then by Theorem \ref{estimate H0}, we have
$$
H_0 =\frac{2}{r_0}+O(r^{-1-\tau}), \quad
X_r \cdot n_0 =r_0 +O(r^{1-\tau}), \quad
K=\frac{1}{r^2_0}+O(r^{-2-\tau}).
$$
Then, by the same arguments as that at page 11 in \cite{FST}, we claim
that
$$
\int_{\Si_r}H_0 d \sigma = 4\pi r_0
+\frac{\Area(\Si_r)}{r_0}+O(r^{1-2\tau}).
$$
In fact, let $K=\frac{1}{r^2 _0}+\bar{K}$, by one of the Minkowski
integral
 formulae  \cite[Lemma 6.2.9]{KLBG},   we have
\begin{equation}\label{BY-ADMH01}
\begin{split}
\int_{\Si_r}H_0d\sigma &=2\int_{\Si_r}KX_r\cdot n_0d\sigma\\
&=2\int_{\Si_r}\lf(\frac1{r_0 ^2}+\bar{K}\ri)X_r\cdot n_0d\sigma\\
&=\frac{2}{r^2 _0}\int_{\Si_r}X_r\cdot n_0d\sigma
+2\int_{\Si_r}\bar{K}
X_r\cdot n_0d\sigma\\
&=\frac{6V(r)}{r_0 ^2}+2\int_{\Si_r}\bar{K}\lf(r_0
+O\lf(r^{1-\tau}\ri)\ri)
d\sigma\\
&=\frac{6V(r)}{r_0 ^2}+2r_0\int_{\Si_r}\bar{K}+O(r^{1-2\tau})\\
&=\frac{6V(r)}{r_0 ^2}+2r_0\int_{\Si_r}\lf(K-\frac1{r_0 ^2}\ri)+O(r^{1-2\tau})\\
&=\frac{6V(r)}{r_0 ^2}+8\pi r_0
-\frac{2\Area(\Si_r)}{r_0}+O(r^{1-2\tau}),
\end{split}
\end{equation}
where $V(r)$ is the volume of the interior of the surface
$X_r(\Sigma_r)$ in $\R^3$. Here and in the sequel, $\Area(\Si_r)$ is
the area of $(\Si, g)$. On the other hand, by the above estimate we
know that  $H_0=\frac2r+H_1$ with $H_1=O\lf(r^{-1-\tau}\ri)$ and by
another Minkowski integral formula we have
\begin{equation}\label{BY-ADMH02}
\begin{split}
2\Area(\Si_r)&= \int_{\Si_r}H_0X_r\cdot n_0d\sigma\\
&=\frac{6V(r)}{r_0}+\int_{\Si_r}H_1X_r\cdot n_0d\sigma\\
&=\frac{6V(r)}{r_0}+r\int_{\Si_r}H_1d\sigma+O\lf(r^{2-2\tau}\ri)\\
&=\frac{6V(r)}{r_0}-2\Area(\Si_r)+r_0 \int_{\Si_r}H_0d\sigma
+O\lf(r^{2-2\tau}\ri),
\end{split}
\end{equation}
which implies
\begin{equation}\label{BY-ADMH03}
\int_{\Si_r}H_0d\sigma=-\frac{6V(r)}{r_0 ^2}+
\frac{4\Area(\Si_r)}{r_0} +O(r^{1-2\tau}).
\end{equation}
 The claim  follows from \eqref{BY-ADMH01} and \eqref{BY-ADMH03}.
 From the claim and (\ref{integral of diff meancurv}) we get
\begin{equation}
\begin{split}
&\int_{\Si_r}(H_0-H) d \sigma=4\pi r_0
+\frac{\Area(\Si_r)}{r_0}-\int_{\Si_r}\hat H d\sigma\\
&+\frac12 \int_{\Si_r}(g_{ij,i}-g_{ii,j})\frac{\p \rho}{\p
x^j}d\sigma_0 + \frac12 \int_{\Si_r}\sigma_{st} \frac{\p^2 \rho}{\p
x^s \p x^t}d\sigma^0 +O(r^{1-2\tau}).
\end{split}
\end{equation}

By the definition, we see that
$$
\frac{\Area(\Si_r)}{r_0}=\frac{\Area_0(\Si_r)}{r_0}+\frac12 r^{-1}_0
\int_{\Si_r}h^{st}\sigma_{st} d\sigma_0 +O(r^{1-2\tau}),
$$
where $\Area_0(\Si_r)$ is the area of $\Si_r$ with respect to induce
metric from $\hat g$.
We also have
$$
\int_{\Si_r}\hat H d\sigma =\int_{\Si_r}\hat H d\sigma_0 +\frac12
\int_{\Si_r}\hat H h^{st}\sigma_{st} d\sigma_0 +O(r^{1-2\tau}).
$$
By Lemma \ref{2ndfundamentalforms} we have
$$
\frac12 \int_{\Si_r}\sigma_{st}\frac{\p^2 \rho}{\p x^s \p
x^t}d\sigma=\frac12\int_{\Si_r}\sigma_{st}\overset{\circ}{\hat
A}_{st} +\frac14 \int_{\Si_r} \hat H h^{st}\sigma_{st}d\sigma_0
+O(r^{1-2\tau})
$$
Combining these things together we get

\begin{equation}
\begin{split}
&\int_{\Si_r}(H_0 -H)d \sigma =4\pi r_0
+\frac{\Area_0(\Si_r)}{r_0}-\int_{\Si_r}\hat H d\sigma_0
+\int_{\Si_r}(\frac{1}{2r_0}-\frac{\hat
H}{4})h^{st}\sigma_{st}d\sigma_0\\
&+ \frac12 \int_{\Si_r}(g_{ij,i}-g_{ii,j})\frac{\p \rho}{\p
x^j}d\sigma_0 +\frac12 \int_{\Si_r}\sigma_{st}B_{st}d\sigma.
\end{split}
\end{equation}

Set $\hat X =y-\overrightarrow{a} $. It is an isometric
embedding of $(\Si_r, \hat g)$ into $\R ^3$. By Proposition
\ref{nearround}, we see that
$$
\hat X\cdot \hat n =r_0 +O(r^{1-\tau}), \quad
\hat H =\frac{2}{r_0}+O(r^{-1-\tau}),\quad
\hat K=\frac{1}{r^2_0}+O(r^{-2-\tau}),
$$
where $\hat K$ is the Gauss curvature of $\Si_r$ in $(M \setminus K,
\hat g)$. Using the same arguments as that at page 11 in \cite{FST} to
$(\Si, \hat g)$ we get
$$
4\pi r_0 +\frac{\Area_0(\Si_r)}{r_0}-\int_{\Si}\hat H d\sigma_0
=O(r^{1-2\tau}).
$$
By Corollary \ref{decay2ndformsinR3} and Proposition \ref{nearround}
we see that
$$
\int_{\Si_r}(\frac{1}{2r_0}-\frac{\hat
H}{4})h^{st}\sigma_{st}d\sigma_0+\frac12
\int_{\Si_r}\sigma_{st}B_{st}d\sigma =O(r^{1-2\tau}).
$$
Thus we have
$$
\int_{\Si_r}(H_0 -H)d \sigma =\frac12
\int_{\Si_r}(g_{ij,i}-g_{ii,j})\frac{\p \rho}{\p x^j}d\sigma_0
+O(r^{1-2\tau}).
$$
The first term in the right hand side of the above equality is the ADM mass
of $(M,g)$. Thus, we finish to prove  Theorem \ref{BYmass}.\end{proof}

\begin{proof}[Proof of Theorem \ref{Hawkingmass}] By Lemma
\ref{difference of mean curvature} and (\ref{3rdterminmeancurv})
we have
\begin{equation}
\begin{split}
\int_{\Si_r}H^2 d \sigma=&\int_{\Si_r}{\hat H}^2 d \sigma
+\int_{\Si_r}H \cdot \hat H \frac{\p \rho}{\p x^i} \frac{\p \rho}{\p
x^j}\sigma_{ij} d \sigma\\
&+\hat{H}\int_{\Si_r} \sigma_{st,i}\frac{\p \rho}{\p x^i}\frac{\p
\rho}{\p x^s}\frac{\p \rho}{\p x^t}d\sigma^0 -2 \hat
{H}\int_{\Si_r}\sigma_{ij}\frac{\p^2 \rho}{\p x^i \p x^j}d\sigma^0
\\
&-2\hat{H}\int_{\Si_r} g_{ij,i}\frac{\p \rho}{\p x^j}d\sigma^0
+\hat{H}\int_{\Si_r} g_{jj,i}\frac{\p \rho}{\p x^i}d\sigma^0 +
O(r^{-2\tau})\\
=&\int_{\Si_r}{\hat H}^2 d \sigma+\int_{\Si_r}H \cdot \hat H
\frac{\p \rho}{\p x^i} \frac{\p \rho}{\p x^j}\sigma_{ij} d
\sigma^0 \\
&-{\hat H}^2 \int_{\Si_r}\sigma_{st}\frac{\p \rho}{\p x^s} \frac{\p
\rho}{\p x^t}d \sigma^0 +\hat{H} \int_{\Si_r}\sigma_{st,t}\frac{\p
\rho}{\p x^s}d \sigma^0 \\
&+\hat {H}\int_{\Si_r}\sigma_{st}\frac{\p^2 \rho}{\p x^s \p
x^t}d\sigma^0
-2\hat {H}\int_{\Si_r}\sigma_{ij}\frac{\p^2 \rho}{\p x^i \p
x^j}d\sigma^0 \\
&-2\hat{H}\int_{\Si_r} g_{ij,i}\frac{\p \rho}{\p x^i}d\sigma^0
+\hat{H}\int_{\Si_r} g_{jj,i}\frac{\p \rho}{\p x^i}d\sigma^0 +
O(r^{-2\tau})\\
=&\int_{\Si_r}{\hat H}^2 d \sigma +\hat{H}\int_{\Si_r}
(g_{jj,i}-g_{ij,j})\frac{\p \rho}{\p x^i}d\sigma^0 \\
&-\hat {H}\int_{\Si_r}\sigma_{ij}\frac{\p^2 \rho}{\p x^i \p
x^j}d\sigma^0 + O(r^{-2\tau}).
\end{split}
\end{equation}
Note that
\begin{equation}
\begin{split}
d\sigma &=(1+h^{ij}\sigma_{ij}+O(r^{-2\tau}))^{\frac12}d\sigma_0
=d\sigma^0 +\frac12 h^{ij}\sigma_{ij} d\sigma_0 +O(r^{-2\tau}).
\end{split}
\end{equation}
Combining above equalities with Lemma \ref{2ndfundamentalforms}, we
get
$$
\int_{\Si_r}H^2 d \sigma =\int_{\Si_r}{\hat H}^2 d \sigma_0
+\hat{H}\int_{\Si_r} (g_{jj,i}-g_{ij,j})\frac{\p \rho}{\p
x^i}d\sigma^0+ O(r^{-2\tau}).
$$
Hence,
\begin{equation}
\begin{split}
m_H (\Si_r)=&\frac{\Area(\Si_r)^{\frac12}}{{(16 \pi)}^\frac32}(16\pi
-\int_{\Si_r}{\hat H}^2 d
\sigma_0)\\
&-\frac{\Area(\Si_r)^{\frac12}}{{(16 \pi)}^\frac32}\cdot
\hat{H}\int_{\Si_r} (g_{jj,i}-g_{ij,j})\frac{\p \rho}{\p
x^i}d\sigma^0
+O(r^{1-2\tau})\\
=&-2\frac{\Area(\Si_r)^{\frac12}}{(16
\pi)^\frac32}\int_{\Si_r}|\overset{\circ}{\hat A}|^2 d\sigma^0
-\frac{\Area(\Si_r)^{\frac12}}{{(16 \pi)}^\frac32}\cdot
\hat{H}\int_{\Si_r} (g_{jj,i}-g_{ij,j})\frac{\p \rho}{\p
x^i}d\sigma^0 \\
&+O(r^{1-2\tau})\\
=&\frac{\Area(\Si_r)^{\frac12}}{{(16 \pi)}^\frac32}\cdot
\hat{H}\int_{\Si_r} (g_{ij,j}-g_{jj,i})\frac{\p \rho}{\p
x^i}d\sigma^0 +O(r^{1-2\tau}),
\end{split}
\end{equation}
where we have used estimate
$$
|\overset{\circ}{\hat A}|\leq C r^{-1-\tau}
$$
in the last equality. On the other hand, we  have

$$
K=\frac{1}{r^2 _0}+O(r^{-2-\tau})
$$
By the Gauss-Bonnet formula, we get
$$
\Area(\Si_r)=4\pi r^2 _0 + O(r^{2-\tau}),
$$
From  Proposition \ref{nearround} we see that
$$
\hat H =\frac{2}{r_0}+O(r^{-1-\tau}),
$$
Combining these formulas we obtain
$$
m_H (\Si_r)=\frac{1}{16 \pi}\int_{\Si_r} (g_{ij,j}-g_{jj,i})\frac{\p
\rho}{\p x^i}d\sigma^0 +O(r^{1-2\tau}).
$$
Thus, we finish to prove the Theorem.
\end{proof}


\begin{thebibliography}{99}

\bibitem{ADM}
Arnowitt, R., Deser, S. and  Misner, C. W., {\it Coordinate
invariance and energy expressions in general relativity}, Phys. Rev.
(2) \textbf{122}, (1961), 997--1006.

\bibitem{BTK86}
Bartnik, R., {\it The mass of an asymptotically flat manifold},
Comm. Pure Appl. Math. \textbf{39} (no. 5), (1986), 661--693.

\bibitem{BLP}
Baskaran, D., Lau S. R. and  Petrov A. N., {\it Center of mass
integral in canonical general relativity}, Ann. Physics \textbf{307}
 (no. 1), (2003),  90--131.

\bibitem{BBWYY}
Braden, H. W., Brown, J. D., Whiting, B. F. and York, J. W., {\it
Charged black hole in a grand canonical ensemble}, Phys. Rev. D (3)
\textbf{42} (no. 10), (1990),  3376--3385.

\bibitem{BLY99}
Brown, J. D., Lau, S. R. and York, J. W., {\it Canonical quasilocal
energy and small spheres}, Phys. Rev. D (3) \textbf{59} (no. 6),
(1999), 064028.

\bibitem{BY1}
Brown, J. D. and York, J. W., {\it Quasilocal energy in general
relativity}, Mathematical aspects of classical field theory
(Seattle, WA, 1991),  Contemp. Math., \textbf{132}, Amer. Math.
Soc., Providence, RI, (1992), 129--142.

\bibitem{BY2}
Brown, J. D. and York, J. W., {\it Quasilocal energy and conserved
charges derived from the gravitational action}, Phys. Rev. D (3)
\textbf{47} (no. 4), (1993), 1407--1419.

\bibitem{Ca} S.M. Carroll
{\it Spacetime and Geometry}, Addison Wesley.


\bibitem{Ch} Sun-Yung, A. Chang
{\it The Moser-Trudinger inequality and applications to some
problems in conformal geometry}, preprint.


\bibitem{CY} D.Christodoulou and S.-T.Yau
{\it Some remarks on the quasi-local mass}, Mathematics and general
relativity (Santa Cruz, CA, 1986)9-14, Contemp.math.71,
Amer.Math.Soc., Providence, RI, 1988.





\bibitem{FST} Xu-Qian Fan, Yuguang Shi, and Luen-Fai Tam,
{\it Large-Sphere and small-sphere limits of the Brown-York mass},
preprint
\bibitem{Herglotz43}
Herglotz, G. {\it \"Uber die Steinersche Formel f\"ur
Parallelfl\"achen},  Abh. Math. Sem. Hansischen Univ. \textbf{15}
(1943), 165--177.
\bibitem{HKG}
Hawking, S. W., {\it Gravitational Radiation in an Expanding
Universe}, J. Math. Phys. \textbf{9}, 1968, 598-604.
\bibitem{HKGH1}
Hawking, S. W. and Horowitz, G. T., {\it The gravitational
Hamiltonian, action, entropy and surface terms}, Classical Quantum
Gravity \textbf{13} (no. 6), (1996),  1487--1498.



\bibitem{HY}
G.Huisken, S.T. Yau, {\it Definition of center of mass for
isolated physical systems and unique foliations by stable spheres
of constant mean curvature,} Invent. Math. 124(1996)281-311.


\bibitem{KLBG}
Klingenberg, W., {\it A course in differential geometry}, Translated
from the German by David Hoffman. Graduate Texts in Mathematics,
Vol. \textbf{51}, Springer-Verlag, New York-Heidelberg, 1978.



\bibitem{LM}
C.De Lellis, S.Muller,  {\it Sharp rigidity estimates for nearly
umbilical surfaces}, J. Differential Geom.69 (2005), 75-110.


\bibitem{Ma}
E. A. Martinez, {\it Quasilocal energy for Kerr black hole},
Physical Review D, 50 (1994), 4920-4929.

\bibitem{Nr}{\it The Weyl and Minkowski problem in differential geometry in the
large}, Communications on Pure and Applied Mathematics, Vol. VI,
337-394(1953).
\bibitem{PAV}
Pogorelov, A.  V., {\it Extrinsic geometry of convex surfaces},
Translated from the Russian by Israel Program for Scientific
Translations. Translations of Mathematical Monographs, Vol.
\textbf{35}, American Mathematical Society, Providence, R.I., 1973.

\bibitem{QT}
Qing J and Tian G, {\it On the uniqueness of the foliation of
spheres of constant mean curvature in asymptotically flat
$3$-maifolds}, J. Am. Math.Soc. 20 (2007), 1091-110.

\bibitem{Sacksteder62}
Sacksteder, R., {\it  The rigidity of
 hypersurfaces}, J. Math. Mech. \textbf{11} (1962), 929--939.
\bibitem{ST02}
Shi, Y.-G. and  Tam, L.-F., {\it Positive mass theorem
 and the boundary
behaviors of compact manifolds with nonnegative scalar curvature},
J. Differential Geom. \textbf{62} (2002), 79--125.

\bibitem{Ye}
Ye R, {\it Foliation by constant mean curvature spheres on
asymptotically flat manifolds}, Geometric analysis and the calculus of
variations (Cambridge, MA: International Press), pp 369-83
(1996)(Preprint:dg-ga/9709020)



\end{thebibliography}
\end{document}